\documentclass[12pt]{amsart}
\usepackage{amsmath,amsthm,xypic,
amssymb, 
enumitem
}

\newtheorem{thm}{Theorem}[section]

\newtheorem{prop}[thm]{Proposition}
\newtheorem{cor}[thm]{Corollary}

\theoremstyle{definition}

\newtheorem*{dff*}{Definition}

\author{M. G. Mahmoudi}

\title[An alternative approach to the concept of separability]{An alternative approach to the concept of separability in
  Galois theory} 

\begin{document}
\maketitle
\begin{abstract}
The notion of a separable extension is an important concept in Galois theory.
Traditionally, this concept is introduced using the minimal polynomial and the formal derivative. 
In this work, we present an alternative approach to this classical concept.  
Based on our approach, we will give new proofs of some basic results about separable extensions (such as the existence of
the separable closure, Theorem of the primitive element and the transitivity of
separability).\\

\end{abstract}

\section*{Introduction}
\label{sec:intro}

In most elementary textbooks on Galois
theory, separable
extensions are usually introduced
via the minimal polynomial: an algebraic extension $E/K$ is
separable if the minimal polynomial $f(X)\in K[X]$ of every element $\alpha\in E$
has nonzero formal derivative, or equivalently $f(X)$ has distinct
roots in its splitting field or any algebraically closed field $\Omega$ containing $E$.

For normal extensions, there is a more functorial definition:
a field extension $E/K$ is normal if there is a unique $K$-algebra embedding of $E$ into $\Omega$ (see \cite[Ch. 6]{lorenz}).
It would be interesting to know if there exists an analogous definition for separable extensions.
The aim of this work is to address this question.

There exist already some methods to introduce separable extensions using the embeddings
of $E/K$ into an algebraic closure $\Omega$ of $E$ (note that every algebraic
closure of $K$ is isomorphic to an algebraic closure of $E$).
Let $\mathrm{Hom}_{K}(E,\Omega)$ denote the set of all $K$-algebra homomorphisms from
$E$ into $\Omega$. 
On can show that an extension $E/K$ of finite degree is separable if and only if 
$|\mathrm{Hom}_{K}(E,\Omega)|=[E:K]$ (see \cite[Ch. 7]{lorenz}).
There is another criterion (see  \cite[Ch. V., \S15,
n. 6]{bourbaki}) as follows: $E/K$ is separable if and only if for
every linearly independent elements $a_1,\cdots,a_n\in E$ over $K$,
there exist $K$-automorphisms $\sigma_1,\cdots,\sigma_n$ of $\Omega$
such that $\det(\sigma_i(a_j))\neq0$.

These properties are very useful characterization of separable
extensions, but applying them are generally less easy.
We suggest using the following alternative definition of
separability:
\begin{dff*}
We say that an element $\alpha$ of algebraic field extension $E/K$ is \emph{separable} if for
every intermediate subfield $L$ of $E/K$ with $\alpha\in E\backslash L$, there
exist two $L$-algebra homomorphisms $\phi,
\psi:E\rightarrow\Omega$ such that $\phi(\alpha)\neq\psi(\alpha)$. 
An algebraic extension $E/K$ is called separable if all its elements are separable.  
\end{dff*}
Other equivalent conditions will be given in Corollary \ref{L1L2} and Corollary \ref{hom>1}.
Roughly speaking, this definition says that $L$ and $\alpha$ can be
\emph{separated} by the homomorphisms from $E$ into $\Omega$. 
First, we show that the above definition is
equivalent to the usual concept of separability.
Then, we give new proofs of some basic results about separable
extensions (such as the existence of the separable closure, Theorem of
the primitive element and the transitivity of separability) based on this approach.

We hope that our approach provides a useful insight into the concept of separability.

\section{Alternative definition of separability}

We recall the following standard facts about the embeddings of a algebraic extension into an algebraically closed field (for the proof see \cite[p. 53 and p. 67]{lorenz}).

\begin{thm}
{\rm (a)} (Tower formula for the number of embeddings)
Let $E/K$ be a field extension of finite degree with an intermediate
subfield $L$. 
Let $\Omega$ be an algebraic closure of $E$.
Then
$$|\mathrm{Hom}_K(E,\Omega)|=|\mathrm{Hom}_L(E,\Omega)|\times|\mathrm{Hom}_K(L,\Omega)|,$$
in particular $|\mathrm{Hom}_K(E,\Omega)|\leqslant[E:K]$. \\
{\rm (b)} (Extending field embeddings) Let $L/F$ be an algebraic extension of fields.
Then every field embedding from $F$ into an algebraically closed field $\Omega$ can be extended to a one from $L$ to $\Omega$.
\end{thm}

\begin{prop}
Let $E/K$ be an extension of fields and $\alpha\in E$.
Let $\Omega$ be the algebraic closure of $E$.
The following statements are equivalent:\\
$(i)$ The minimal polynomial of $\alpha$ over $K$ has distinct roots in $\Omega$.\\
$(ii)$ For
every intermediate subfield $L$ of $E/K$ with $\alpha\in E\backslash L$, there
exist $L$-algebra homomorphisms $\phi,
\psi:E\rightarrow\Omega$ such that $\phi(\alpha)\neq\psi(\alpha)$.\\
$(iii)$ For every intermediate subfield $L$ of $K(\alpha)/K$ with
$\alpha\in K(\alpha)\backslash L$, there
exist two $L$-algebra homomorphisms $\phi,
\psi:E\rightarrow\Omega$ such that $\phi(\alpha)\neq\psi(\alpha)$.
\end{prop}

\begin{proof}
  $(i)\Rightarrow(ii)$
 The minimal polynomial of $\alpha$ over $L$ has
  distinct roots in $\Omega$, so $L(\alpha)$ has distinct
  $L$-embeddings into $\Omega$.
  $(ii)\Rightarrow(iii)$ is immediate.
  $(iii)\Rightarrow(i)$
Let $f(X)\in K[X]$ be the minimal polynomial of $\alpha$.
If the minimal polynomial of $\alpha$ over $K$ does not have distinct roots in $\Omega$, then $K$ is of characteristic $p>0$ and there exists
$n\geqslant 1$ such that
$f(X)=g(X^{p^n})$ where $g(X)\in K[X]$ has nonzero formal derivative.
Take $L=K(\alpha^{p^{n}})$. 
Note that $\alpha\in K(\alpha)\backslash L$, otherwise $\alpha$ would satisfy in a polynomial of degree less than $\deg f(x)$.
Since $\alpha^{p^n}\in L$, we have
$\phi(\alpha)=\psi(\alpha)$ for every
$L$-algebra homomorphism $\phi, \psi:E\rightarrow\Omega$, contradiction.
\end{proof}

  The above equivalence, in particular shows that the definition of
  separability given in (ii), which at first seems \emph{relative},
  does not depend on $E$. 
Thus, if $E$ and $E'$ are two
  algebraic field extension of $K$ containing $\alpha$, then $\alpha$
 is separable as an element of $E$ then it is separable as an element of $E'$.
Also, if $K\subset L\subset E$ and $\alpha\in E$ is separable over $K$,
then it is separable over $L$.

\begin{cor}\label{L1L2}
Let $E/K$ be a separable extension with intermediate subfields $L_1$
and $L_2$.
The following conditions are equivalents:\\
$(i)$ $L_1\subset L_2$.\\
$(ii)$ For every $\phi, \psi\in\mathrm{Hom}_K(E,\Omega)$,
$\phi|_{L_2}=\psi|_{L_2}$ implies $\phi|_{L_1}=\psi|_{L_1}$.\\
Conversely if $(i)$ and $(ii)$ are equivalents for every $L_1$ and
$L_2$ then $E/K$ is separable.
\end{cor}

\begin{proof}
$(ii)\Rightarrow(i)$: If $L_1\not\subset L_2$, consider an element
$\alpha\in L_1\backslash L_2$, then there exists $\phi,
\psi\in\mathrm{Hom}_{L_2}(E,\Omega)$ with
$\phi(\alpha)\neq\psi(\alpha)$.
This contradicts $(ii)$.
The implication $(i)\Rightarrow(ii)$ is evident.\\
Conversely suppose that $L$ is an intermediate subfield and $\alpha\in
E\backslash L$.
If there exist no $\phi, \psi\in\mathrm{Hom}_{L}(E,\Omega)$ with
$\phi(\alpha)\neq\psi(\alpha)$, then by taking $L_1=K(\alpha)$ and
$L_2=L$, the equivalence of $(i)$ and $(ii)$ implies that
$K(\alpha)\subset L$, contradiction.
\end{proof}

\begin{cor}
  Let $E/K$ be a separable extension and let $\alpha, \beta\in E$.
The following conditions are equivalent:\\
$(i)$ $\alpha\in K(\beta)$.\\
$(ii)$ For every $\phi, \psi\in\mathrm{Hom}_K(E,\Omega)$,
$\phi(\beta)=\psi(\beta)$ implies $\phi(\alpha)=\psi(\alpha)$.
\end{cor}

As a consequence we obtain
\begin{cor}\label{hom>1}
  An algebraic field extension $E/K$ is separable if and only if for
  every proper intermediate subfield $L$ of $E/K$, $|\mathrm{Hom}_L(E,\Omega)|>1$.
\end{cor}

\begin{proof}
Consider an element $\alpha\in E\backslash L$.
As $E/K$ is separable, there exists $\phi,
\psi\in\mathrm{Hom}_L(E,\Omega)$ such that
$\phi(\alpha)\neq\psi(\alpha)$, hence $|\mathrm{Hom}_L(E,\Omega)|>1$.
Conversely, suppose that $|\mathrm{Hom}_L(E,\Omega)|>1$ for every
intermediate subfield $L$. 
Consider an element $\alpha\in E\backslash L$. 
We show that there exist $\phi, \psi\in\mathrm{Hom}_L(E,\Omega)$ such
that $\phi(\alpha)\neq\psi(\alpha)$. 
If it is not the case, by taking $L_2=L$ and $L_1=K(\alpha)$, Corollary \ref{L1L2} implies
that $K(\alpha)\subset L$, contradiction.
\end{proof}

\section{Some applications}
\label{sec:some-applications}

\begin{thm}[Primitive Element Theorem]
\label{primitive-element}
Let $E/K$ be a separable field extension of finite degree.
Then there
exists an element $\alpha\in E$ such that $E=K(\alpha)$.
\end{thm}

\begin{proof}
For finite fields, we follow the standard argument.
If $K$ is a field with $q$ elements and $|E|=q^n$, then we have
$|L|\leqslant q^{n-1}$ for every proper intermediate subfield $L$ of
$E/K$, hence every element $x\in E$ satisfies $x^{q^{n-1}}-x=0$.
Since $|E|>q^{n-1}$,
there exists certainly an element $\alpha\in E$ with
$\alpha^{q^{n-1}}-\alpha\neq0$, this element is not included in any proper subfield, hence $K(\alpha)=E$.

Now consider the case where $K$ is infinite.
Let $\alpha, \beta\in E$. 
It suffices to prove that there exists $\gamma\in E$ such that
$\alpha, \beta\in K(\gamma)$.
We may assume that $\alpha$ and $\beta$ are linearly independent over
$K$.
Let $P$ be the plane spanned by $\alpha$ and $\beta$.
If our claim is not true then for every nonzero $\gamma\in P$,
$K(\gamma)\cap P$ is one dimensional over $K$.
Since $K$ is infinite, there exists an infinite number of pairwise noncolinear elements $\gamma_1, \gamma_2,
\cdots\in P$.
It follows that $\gamma_i\not\in K(\gamma_j)$ if $i\neq j$.
By Corollary \ref{L1L2}, 
there exist 
$\phi_{ij}, \psi_{ij}\in
\mathrm{Hom}_K(E,\Omega)$ such that $\phi_{ij}|_{K(\gamma_i)}=\psi_{ij}|_{K(\gamma_i)}$ but $\phi_{ij}|_{K(\gamma_j)}\neq\psi_{ij}|_{K(\gamma_j)}$.
As $\mathrm{Hom}_K(E,\Omega)$ is finite, there exist distinct $i, j,
k$ such that $\phi_{ji}=\phi_{ki}$ and $\psi_{ji}=\psi_{ki}$.
We obtain 
$$\phi_{ki}(\gamma_j)=\phi_{ji}(\gamma_j)=\psi_{ji}(\gamma_j)=\psi_{ki}(\gamma_j),$$
on the other hand 
  $\phi_{ki}(\gamma_k)=\psi_{ki}(\gamma_k),$
so $\phi_{ki}|_P=\psi_{ki}|_P$, contradiction.
\end{proof}

\begin{cor}\label{K(alpha)}
  If $E/K$ is a separable extension of finite degree then
  $|\mathrm{Hom}_K(E,\Omega)|=[E:K]$.
Conversely if $|\mathrm{Hom}_K(E,\Omega)|=[E:K]$ then $E/K$ is
separable. 
In particular if $\alpha\in E$ is separable over $K$ then the
extension $K(\alpha)/K$ is separable.
\end{cor}

\begin{proof}
Suppose that $[E:K]=n$.
  By Theorem \ref{primitive-element}, there exists $\alpha\in E$ with
  $E=K(\alpha)$.
As $\alpha$ is separable over $F$, the minimal polynomial $f(x)\in
K[x]$ of $\alpha$ has $n$ distinct roots
$\alpha_1,\cdots,\alpha_n\in\Omega$.
For every $i=1,\cdots,n$ we can define a map
$\phi_i\in\mathrm{Hom}(E,\Omega)$ by $\phi_i(\alpha)=\alpha_i$.
Conversely for every $\phi\in\mathrm{Hom}_K(E,\Omega)$, $\phi(\alpha)$
is necessarily equal to one of the elements
$\alpha_1,\cdots,\alpha_n$.

Conversely, suppose that $|\mathrm{Hom}_K(E,\Omega)|=[E:K]$. 
We have to show that $E/K$ is separable. 
Otherwise, by Corollary \ref{hom>1}, there exists a proper intermediate subfield $L$ such that
$|\mathrm{Hom}_L(E,\Omega)|=1$.
Thus we have
$|\mathrm{Hom}_K(E,\Omega)|=|\mathrm{Hom}_K(L,\Omega)|\times|\mathrm{Hom}_L(E,\Omega)|=|\mathrm{Hom}_K(L,\Omega)|\leqslant[L:K]<n$,
contradiction.

\end{proof}

\begin{prop}\label{sum_sep}
  Let $E/K$ be an algebraic extension. 
Then the set of separable elements of $E/K$ form an intermediate subfield.
\end{prop}

\begin{proof}
 It suffices to prove that for every separable elements $\alpha,
 \beta\in E$, $\alpha+\beta$ and $\alpha\beta$ is separable. 
We may assume that $\alpha\beta\neq0$.
For an intermediate subfield $L$ with $\alpha+\beta\not\in L$,
we should prove that there exists two maps in
$\mathrm{Hom}_L(E,\Omega)$ whose values on $\alpha+\beta$ are different.
If $\alpha\in L$, then by separability of $\beta$ there exists $\phi,
\psi\in \mathrm{Hom}_L(E,\Omega)$ with $\phi(\beta)\neq\psi(\beta)$,
hence $\phi(\alpha+\beta)\neq\psi(\alpha+\beta)$.
We may then assume that $\alpha, \beta\not\in L$.
If $\alpha\not\in L(\beta)$, then by separability of $\alpha$, there
exist
$\phi,\psi\in\mathrm{Hom}_{L(\beta)}(E,\Omega)\subseteq\mathrm{Hom}_L(E,\Omega)$
with $\phi(\alpha)\neq\psi(\alpha)$ and we obtain
$\phi(\alpha+\beta)\neq\psi(\alpha+\beta)$.
So assume that $\alpha\in L(\beta)$.
As $\beta$ is separable over $K$, it is also separable over $L$, hence
by Corollary \ref{K(alpha)}
every element of $L(\beta)$ is separable over $L$, in particular
$\alpha+\beta$ is separable over $L$.
Hence $|\mathrm{Hom}_L(L(\alpha+\beta),\Omega)|>1$ and there exist
$\phi, \psi\in\mathrm{Hom}_L(L(\alpha+\beta),\Omega)$ with
$\phi(\alpha+\beta)\neq\psi(\alpha+\beta)$, $\phi$ and $\psi$ can be
extended to $\bar{\phi}, \bar{\psi}\in\mathrm{Hom}_L(E,\Omega)$ and we
are done.
The proof for $\alpha\beta$ is similar.
\end{proof}

\begin{cor}
  Let $E/L$ and $L/K$ be algebraic separable extensions.
Then $E/K$ is separable as well.
\end{cor}

\begin{proof}
  Let $M$ be an intermediate subfield of $E/K$ and $\alpha\in
  E\backslash M$.
We should prove the existence of two homomorphisms $\phi,
\psi\in\mathrm{Hom}_M(E,\Omega)$ with $\phi(\alpha)\neq\psi(\alpha)$.
If $\alpha\not\in LM$, then by the separability of $E/L$, there exist
$\phi,
\psi\in\mathrm{Hom}_{LM}(E,\Omega)\subset\mathrm{Hom}_M(E,\Omega)$
with $\phi(\alpha)\neq\psi(\alpha)$ and we are done.
Consider the case where, $\alpha\in LM$. 
As all elements of $L$ are separable over $K$, then by
Proposition \ref{sum_sep}, $LM/M$ is separable.
There exist so two homomorphisms $\phi,
\psi\in\mathrm{Hom}_M(LM,\Omega)$ with $\phi(\alpha)\neq\psi(\alpha)$.
These homomorphisms can be extended to homomorphisms in $\mathrm{Hom}_M(E,\Omega)$.
\end{proof}

\bibliographystyle{myamsplain}
\vspace{.25cm}

\tiny
\noindent{M. G. Mahmoudi,
  {\tt
    mmahmoudi@sharif.ir}, \\
  Department of Mathematical Sciences, Sharif University of
  Technology, P. O. Box 11155-9415, Tehran, Iran.}

\end{document}